\documentclass[11pt,a4paper]{article}
\usepackage[T1]{fontenc}
\usepackage[utf8]{inputenc}
\usepackage{lmodern}
\usepackage{microtype}
\usepackage[margin=28mm]{geometry}
\usepackage{amsmath,amssymb,amsthm,mathtools}
\usepackage{enumitem}
\usepackage{xcolor}
\usepackage{hyperref}
\usepackage{bookmark}
\hypersetup{colorlinks=true,linkcolor=blue!45!black,citecolor=blue!45!black,
  urlcolor=blue!45!black,
  pdftitle={},pdfauthor={},pdfsubject={},pdfkeywords={},
  pdfcreator={},pdfproducer={}}
\ifdefined\pdfinfoomitdate\pdfinfoomitdate=1\fi
\ifdefined\pdftrailerid\pdftrailerid{}\fi
\ifdefined\pdfsuppressptexinfo\pdfsuppressptexinfo=15\fi
\numberwithin{equation}{section}
\newtheorem{theorem}{Theorem}[section]
\newtheorem{proposition}[theorem]{Proposition}
\newtheorem{lemma}[theorem]{Lemma}
\newtheorem{corollary}[theorem]{Corollary}
\theoremstyle{remark}

\DeclareMathOperator{\Li}{Li}
\DeclareMathOperator{\rad}{rad}
\newcommand{\OO}{\mathcal O}
\newcommand{\PP}{\mathcal P}
\newcommand{\WW}{\mathcal W}
\newcommand{\SeqA}{\mathcal A}
\newcommand{\BB}{\mathcal B}
\newcommand{\one}{\mathbf 1}
\newcommand{\fa}{\mathfrak a}
\newcommand{\fb}{\mathfrak b}
\newcommand{\ff}{\mathfrak f}
\newcommand{\fm}{\mathfrak m}
\newcommand{\fl}{\mathfrak l}
\newcommand{\fq}{\mathfrak q}
\setlist{nosep}
\title{Divisibility Estimates for Gaussian Norms and\\ Almost-Prime Orders of a CM Elliptic Curve}
\author{SeungJu Lee}
\date{\today}

\begin{document}
\maketitle
\begin{abstract}
Let \(E:y^2=x^3-x\). We prove that
\(\#E(\mathbf F_p)/8\) is squarefree and has at most two prime factors
for \(\gg X/(\log X)^2\) primes \(p\le X\).
The proof follows the Gaussian CM sieve framework of Iwaniec and
Jim\'enez Urroz and incorporates the intersection terms identified
by Xie. For a squarefree Gaussian ideal \(\fa\), we
separate its largest rational ideal factor and use the resulting
divisibility of both integral coordinates to control the large-factor
tail. This transfers an ideal Bombieri--Vinogradov estimate to
rational norm divisibility at level \(1/2\), with any fixed power
of the divisor function as a weight. We prove the corresponding
distribution estimate
for the switched product sequence and apply the transfer to both
sequences in Chen's weighted sieve. The primes counted may be
restricted to \(p\equiv5\pmod8\); their normalized orders are odd,
squarefree, and supported on primes congruent to \(1\pmod4\).
\end{abstract}

\section{Introduction}

The study of prime and almost-prime orders of reductions of elliptic
curves is motivated by Koblitz's conjecture \cite{Koblitz1988}.
An early unconditional result for CM curves is Cojocaru's \(P_5\)
theorem for curves with maximal CM order and trivial rational torsion
throughout their rational isogeny class \cite[Theorem~4]{Cojocaru2005}.
Here \(P_r\) denotes a positive integer with at most \(r\) prime factors,
counted with multiplicity.

We consider the fixed elliptic curve
\[
 E:\ y^2=x^3-x.
\]
For a prime \(p\equiv1\pmod4\), its order is divisible by \(8\), and
we write
\[
 n_p=\frac{\#E(\mathbf F_p)}8.
\]
The factor \(8\) arises from the primary Frobenius congruence in
\(\mathbf Z[i]\), which will also give \(n_p\) as a Gaussian norm.

\begin{theorem}\label{thm:main}
There is a constant \(c_E>0\) such that, for all sufficiently large
\(X\),
\[
 \#\left\{p\le X:
 \begin{array}{l}
 p\text{ prime},\quad p\equiv5\pmod8,\\
 n_p\text{ squarefree},\quad \Omega(n_p)\le2
 \end{array}\right\}
 \ge c_E\frac{X}{(\log X)^2}.
\]
The primes counted may be chosen so that every prime divisor of
\(n_p\) is at least \(X^{1/8}\). Here \(\Omega(n)\) counts prime
factors with multiplicity, with \(\Omega(1)=0\).
\end{theorem}

\subsection{Previous work and norm divisibility}

Iwaniec and Jim\'enez Urroz \cite[Theorem~1.2]{IwaniecJimenezCM}
stated the squarefree \(P_2\) lower bound for \(n_p\), with
\(p\equiv1\pmod4\). Their proof combines the linear sieve with
Chen's switching principle in \(\mathbf Z[i]\).
Jim\'enez Urroz \cite[Theorem~1 and \S4]{JimenezUrroz2008}
formulated a squarefree \(P_2\) result for CM elliptic curves over
\(\mathbf Q\) with maximal CM order, at split primes and after
dividing the group order by a curve-dependent fixed divisor.

Xie \cite[Remark~4.1, p.~28]{XieCM2025} identified an omitted
intersection in the first displayed equality on p.~823 of
\cite{IwaniecJimenezCM}. If a rational prime \(\ell\equiv1\pmod4\)
splits as \((\ell)=\fl\overline{\fl}\), then
\begin{equation}\label{eq:split-union}
 \one_{\ell\mid N\alpha}
 =\one_{\fl\mid(\alpha)}
  +\one_{\overline{\fl}\mid(\alpha)}
  -\one_{(\ell)\mid(\alpha)}.
\end{equation}
The corresponding remainder is therefore
\[
 R(\fl)+R(\overline{\fl})-R((\ell)).
\]
The reindexing on that page omits the last term. The rational local
density in the same paper includes its main-term
counterpart,
\[
 g(\ell)=\frac{2}{\ell-1}-\frac{1}{(\ell-1)^2}.
\]
Thus the issue lies in the summation of remainders. It affects the
derivation of (3.8) and (3.9) there, and the analogous conversion
for the switched sequence in (5.9), p.~830.
For example, with \(b=2(1+i)\), the primary Gaussian prime
\(\pi=-9+10i\) has norm \(181\) and
\((\pi-1)/b=5i\). Both conjugate ideals above \(5\) divide this
quotient, so the intersection in \eqref{eq:split-union} occurs
in the sequence under consideration.

Xie's remainder formula
\cite[(4.13)]{XieCM2025} retains all squarefree ideals
\(\fa\) with \(\rad(N\fa)=d\). Such ideals satisfy
\(d\le N\fa\le d^2\). Applying an ideal distribution estimate
up to norm \(X^{1/2}/(\log X)^B\) through this last upper bound
gives rational moduli up to \(X^{1/4}\), with a logarithmic loss.
This yields the squarefree \(P_5\) result in
\cite[Corollary~4.2]{XieCM2025}. We use the factorization below
to estimate the terms of large ideal norm and obtain the
rational norm-divisibility estimate at level \(1/2\).

\subsection{The distribution estimates}

The central step is a transfer theorem for Gaussian norms.
A squarefree ideal has a unique decomposition
\[
 \fa=(c)\fb,\qquad
 N\fa=c^2N\fb,\qquad
 \rad(N\fa)=cN\fb,
\]
where \(c\) is a positive odd squarefree rational integer and \(N\fb\)
is squarefree. If \(\fa\mid(\alpha)\), then \(c\) divides both
coordinates of \(\alpha\).
For \(c\le(\log X)^L\), an ideal norm cutoff costs only a fixed
power of \(\log X\). For larger \(c\), lattice counting and
divisor moments give an arbitrarily strong logarithmic saving
after choosing \(L\).
Theorem~\ref{thm:transfer} consequently retains level \(1/2\),
including weights \(\tau(d)^J\) for every fixed integer \(J\ge0\)
on rational moduli.

We apply this result first to
\[
 N\left(\frac{\pi-1}{2(1+i)}\right)
\]
for primary Gaussian primes \(\pi\), and then to \(N(1+\omega)\),
where \(\omega\) is a prescribed product of three primary
Gaussian primes times \(2(1+i)\) and a unit.
The distribution of the product sequence is established in
Theorem~\ref{thm:switching} using Gaussian prime distribution
and the multiplicative large sieve.
These two rational-divisibility estimates supply the two
distribution inputs for the weighted sieve.
The \(P_2\) conclusion, its local densities, and the CM form of
Chen's weight already appear in \cite{IwaniecJimenezCM}.
Here we establish the norm transfer and apply it to both
sequences.
In a primitive Gaussian setting, an ideal-to-norm sieve passage
also appears in \cite[Lemma~5.3 and \S5.1]{Merikoski2025}.
The transfer below permits nontrivial rational content and
controls the additional intersection terms.

\section{Divisibility of Gaussian norms}

Put \(\OO=\mathbf Z[i]\), \(N\alpha=\alpha\overline\alpha\), and
let \(N\fa=[\OO:\fa]\) for a nonzero integral ideal.
Write \(\Phi(\fa)=|(\OO/\fa)^\times|\),
\(\tau(n)=\sum_{d\mid n}1\), and denote the rational and ideal
M\"obius functions by \(\mu\) and \(\mu_{\OO}\), respectively.
Sums containing the factor \(\mu(d)^2\) are understood to run
only over squarefree positive integers \(d\).
An ideal is called odd if it is coprime to \((1+i)\).
For a squarefree ideal define
\[
 d(\fa)=\rad(N\fa),\qquad c(\fa)=\frac{N\fa}{d(\fa)}.
\]
Squarefreeness of an ideal refers to its prime ideal factors;
both conjugate factors above a split rational prime are allowed.

\begin{lemma}\label{lem:content}
Every squarefree ideal \(\fa\) has a unique factorization
\[
 \fa=(c)\fb,
\]
where \(c=c(\fa)\) is a positive odd squarefree rational integer,
\((c)\) and \(\fb\) are coprime, and \(N\fb\) is squarefree.
Moreover,
\[
 d(\fa)=cN\fb,\qquad N\fa=c^2N\fb.
\]
If \(\fa\mid(\alpha)\), then
\(c\mid\gcd(\Re\alpha,\Im\alpha)\).
\end{lemma}

\begin{proof}
At a split odd prime, include \(\ell\) in \(c\) if both prime
ideals above \(\ell\) divide \(\fa\); if exactly one of them
divides \(\fa\), place that prime ideal in \(\fb\).
An inert prime ideal \((\ell)\), of norm \(\ell^2\), contributes
\(\ell\) to \(c\) when it divides \(\fa\).
The possible single factor \((1+i)\) is placed in \(\fb\).
Unique factorization of ideals proves all the assertions.
\end{proof}

\begin{lemma}\label{lem:content-tail}
For fixed \(C>0\) and integer \(J\ge0\), there is a constant
\(L_J\) such that, for \(X\ge3\) and \(Z\ge1\),
\[
 \sum_{\substack{\alpha\in\OO,\ 0<N\alpha\le CX\\
                  \gcd(\Re\alpha,\Im\alpha)>Z}}
       \tau(N\alpha)^J
 \ll_{C,J} XZ^{-1/2}(\log X)^{L_J}.
\]
The greatest common divisor is positive, including when one
coordinate is zero.
\end{lemma}

\begin{proof}
For \(R\ge1\), a disk of radius \(R\) contains \(O(R^2)\)
nonzero Gaussian integers. Hence the number of elements in
the indicated sum is at most
\[
 \sum_{Z<g\le\sqrt{CX}}
       \#\{\alpha:0<N\alpha\le CX,\ g\mid\alpha\}
 \ll_C X\sum_{g>Z}g^{-2}\ll_C X/Z.
\]
Write \(r_2(n)=\#\{\alpha\in\OO:N\alpha=n\}\) for \(n\ge1\).
The sum-of-two-squares formula gives \(r_2(n)\le4\tau(n)\).
For any fixed nonnegative integer \(t\), it follows that
\begin{equation}\label{eq:divisor-moment}
 \sum_{0<N\alpha\le CX}\tau(N\alpha)^t
 \le4\sum_{n\le CX}\tau(n)^{t+1}
 \ll_{C,t}X(\log X)^{L'_t}.
\end{equation}
One may use \(\tau(n)^k\le\tau_{2^k}(n)\) and the elementary
bound \(\sum_{n\le U}\tau_j(n)\ll_j U(1+\log U)^{j-1}\).
Here \(U\ge1\), and \(\tau_j(n)\) counts ordered \(j\)-tuples
of positive integers with product \(n\).
The first inequality follows prime by prime by encoding the
exponent levels of a \(k\)-tuple of divisors by the subsets
containing each level. Cauchy's inequality now proves the lemma.
\end{proof}

\begin{theorem}[Transfer at level \(1/2\)]\label{thm:transfer}
For \(X\ge3\), let \(m_X(\alpha)\) be a nonnegative weight on
nonzero Gaussian integers, supported on \(N\alpha\le CX\), with
\[
 m_X(\alpha)\ll\tau(N\alpha)^M
\]
for fixed \(C>0\) and a nonnegative integer \(M\). Suppose
\[
 A_X(\fa):=\sum_{\fa\mid(\alpha)}m_X(\alpha)
       =Y_X\gamma_X(\fa)+R_X(\fa),
\]
where \(0\le Y_X\ll X(\log X)^M\), \(\gamma_X(\fa)\ge0\), and
\[
 \gamma_X(\fa)\ll\frac{\tau(N\fa)^{B_0}}{N\fa}
\]
for a fixed nonnegative integer \(B_0\).
Assume that for every \(A>0\) some
\(B=B(A)\) satisfies
\begin{equation}\label{eq:ideal-hypothesis}
 \sum_{N\fa\le\sqrt X/(\log X)^B}|R_X(\fa)|
 \ll_A\frac{X}{(\log X)^A}.
\end{equation}
For squarefree positive integers \(d\), put
\[
 A_X[d]=\sum_{d\mid N\alpha}m_X(\alpha),\qquad
 g_X(d)=\mu(d)\sum_{d(\fa)=d}
                \mu_{\OO}(\fa)\gamma_X(\fa),
\]
where the inner sum is over squarefree ideals.
Then for every \(A>0\) and integer \(J\ge0\) some \(B=B(A,J)\)
satisfies
\begin{equation}\label{eq:transfer-conclusion}
 \sum_{d\le\sqrt X/(\log X)^B}\mu(d)^2\tau(d)^J
       |A_X[d]-Y_Xg_X(d)|
 \ll_{A,J}\frac{X}{(\log X)^A}.
\end{equation}
Constants may depend on the fixed parameters in the hypotheses.
\end{theorem}

\begin{proof}
First we upgrade \eqref{eq:ideal-hypothesis} to include any
fixed weight \(\tau(N\fa)^J\). For \(1\le Q\le\sqrt X\),
Cauchy's inequality gives
\begin{align*}
 \sum_{N\fa\le Q}\tau(N\fa)^J|R_X(\fa)|
 &\le
 \left(\sum_{N\fa\le Q}|R_X(\fa)|\right)^{1/2}
 \left(\sum_{N\fa\le Q}\tau(N\fa)^{2J}|R_X(\fa)|\right)^{1/2}.
\end{align*}
For the second factor use \(|R_X|\le A_X+Y_X\gamma_X\).
The number of ideal divisors of \((\alpha)\) is at most
\(\tau(N\alpha)^2\), and \(N\fa\mid N\alpha\) if
\(\fa\mid(\alpha)\). Interchanging sums and using
\eqref{eq:divisor-moment} therefore gives
\[
 \sum_{N\fa\le Q}\tau(N\fa)^{2J}A_X(\fa)
 \ll\sum_{0<N\alpha\le CX}\tau(N\alpha)^{M+2J+2}
 \ll X(\log X)^{C_1}.
\]
There are at most \(\tau(n)\) Gaussian ideals of norm \(n\).
The main-term contribution is consequently bounded by
\[
 Y_X\sum_{n\le Q}\frac{\tau(n)^{2J+B_0+1}}n
 \ll X(\log X)^{C_2}.
\]
Using \eqref{eq:ideal-hypothesis} with a larger logarithmic
saving proves its divisor-weighted form.

For squarefree \(d\), local inclusion--exclusion gives the exact identity
\begin{equation}\label{eq:exact-transfer}
 \one_{d\mid N\alpha}
 =\mu(d)\sum_{d(\fa)=d}
              \mu_{\OO}(\fa)\one_{\fa\mid(\alpha)}.
\end{equation}
At a split prime this is \eqref{eq:split-union}; at an inert
or ramified prime there is a single ideal-divisibility condition.
Multiplication over the prime factors of \(d\) proves
\eqref{eq:exact-transfer}. Thus
\[
 A_X[d]-Y_Xg_X(d)
 =\mu(d)\sum_{d(\fa)=d}\mu_{\OO}(\fa)R_X(\fa).
\]

Set \(Z=(\log X)^L\) and \(D=\sqrt X/(\log X)^B\).
For terms with \(c(\fa)\le Z\) and \(d(\fa)\le D\),
Lemma~\ref{lem:content} gives \(N\fa\le DZ\).
Since \(\tau(d(\fa))\le\tau(N\fa)\), the weighted ideal estimate
bounds their total contribution after increasing \(B\) by \(L\).

For the terms with \(c(\fa)>Z\), bound the remainder by its
counting and main-term parts. For the former, interchange sums.
Every contributing \(\alpha\) has
\(\gcd(\Re\alpha,\Im\alpha)>Z\).
There are at most \(\tau(N\alpha)^2\) incident ideals, and each
satisfies \(\tau(d(\fa))^J\le\tau(N\alpha)^J\). Hence
\begin{align*}
 \sum_{\substack{d(\fa)\le D\\c(\fa)>Z}}
       \tau(d(\fa))^J A_X(\fa)
 &\ll
 \sum_{\substack{0<N\alpha\le CX\\
                  \gcd(\Re\alpha,\Im\alpha)>Z}}
       \tau(N\alpha)^{M+J+2}\\
 &\ll XZ^{-1/2}(\log X)^{C_3}
\end{align*}
by Lemma~\ref{lem:content-tail}.
For the main term, write \(\fa=(c)\fb\) and \(n=N\fb\).
The hypotheses and the bound for the number of ideals of norm
\(n\) imply
\begin{align*}
 \sum_{\substack{d(\fa)\le D\\c(\fa)>Z}}
       \tau(d(\fa))^J\gamma_X(\fa)
 &\ll \sum_{c>Z}\frac{\tau(c)^{J+2B_0}}{c^2}
           \sum_{n\le D/c}\frac{\tau(n)^{J+B_0+1}}n\\
 &\ll Z^{-1/2}(\log X)^{C_4}.
\end{align*}
Indeed, for fixed \(t\),
\[
 \sum_{c>Z}\frac{\tau(c)^t}{c^2}
 \le Z^{-1/2}\sum_{c\ge1}\frac{\tau(c)^t}{c^{3/2}}
 \ll_t Z^{-1/2}.
\]
The main-term tail after multiplication by \(Y_X\) is therefore
\(\ll XZ^{-1/2}(\log X)^{M+C_4}\).
Choose \(L\) sufficiently large for both tails, then \(B\)
sufficiently large for the small-content estimate.
This proves \eqref{eq:transfer-conclusion}.
\end{proof}

\section{Distribution of the normalized CM orders}

Fix \(b=2(1+i)\), so \(Nb=8\).
A primary Gaussian prime in this paper is a prime element
\(\pi\equiv1\pmod b\) whose norm is a rational prime.
Let \(\PP\) be their set, and write
\[
 \Pi(T)=\#\{\pi\in\PP:N\pi\le T\}
       =2\pi(T;4,1)\sim\Li(T).
\]
In particular, both primary conjugates above a split rational
prime \(p\) are counted. For \(E:y^2=x^3-x\), these are the
two Frobenius eigenvalues of its reduction at \(p\), and the
CM point-count formula gives
\begin{equation}\label{eq:cm-norm}
 \#E(\mathbf F_p)=N(\pi-1),\qquad
 n_p=N\left(\frac{\pi-1}{b}\right)
 \quad(N\pi=p).
\end{equation}
See \cite[\S1]{IwaniecJimenezCM} for this normalization.

\subsection{The prime-distribution input}

We use Huxley's Gaussian ray-class Bombieri--Vinogradov theorem
in the prime-ideal formulation stated in
\cite[Theorem~4]{WongCM2026}.
If \(h(\fq)\) is the ray-class number and
\(\pi_{\mathbf Q(i)}(t;\fq,\mathcal C)\) counts prime ideals of norm at most \(t\)
in the ray class \(\mathcal C\), this formulation gives, for
every \(A>0\) and suitable \(B=B(A)\),
\begin{equation}\label{eq:ray-bv}
 \sum_{N\fq\le\sqrt X/(\log X)^B}
 \frac{h(\fq)}{\Phi(\fq)}
 \max_{\mathcal C}\max_{2\le t\le X}
 \left|\pi_{\mathbf Q(i)}(t;\fq,\mathcal C)
             -\frac{\Li(t)}{h(\fq)}\right|
 \ll_A\frac{X}{(\log X)^A}.
\end{equation}
For \(\fq=b\fa\), all four units are distinguished modulo
\(\fq\). Since \(\mathbf Z[i]\) has class number one,
\[
 h(b\fa)=\frac{\Phi(b\fa)}4.
\]
Thus the weight \(h/\Phi\) in \eqref{eq:ray-bv} is \(1/4\)
on these moduli.
The degree-two prime ideals counted up to \(X\) have norm \(\ell^2\) and number
\(O(\sqrt X)\). Deleting them changes each residue count by
at most this amount. Since there are \(O(Q)\) Gaussian ideals
of norm at most \(Q\), the total deletion error at
\(Q=\sqrt X/(\log X)^B\) is \(O(X/(\log X)^B)\).
It can be absorbed by increasing \(B\).

\begin{proposition}\label{prop:first-distribution}
For a squarefree positive integer \(d\), let
\[
 A_d(X)=2\#\{p\le X:p\text{ prime},\
                   p\equiv1\pmod4,\ d\mid n_p\}.
\]
Define a multiplicative function \(g\) on squarefree integers by
\begin{equation}\label{eq:density}
 g(2)=\frac12,\qquad
 g(\ell)=
 \begin{cases}
 \displaystyle\frac2{\ell-1}-\frac1{(\ell-1)^2},
       &\ell\equiv1\pmod4,\\[4pt]
 \displaystyle\frac1{\ell^2-1},&\ell\equiv3\pmod4.
 \end{cases}
\end{equation}
For every \(A>0\) and integer \(J\ge0\), some \(B=B(A,J)\)
satisfies
\begin{equation}\label{eq:first-distribution}
 \sum_{d\le\sqrt X/(\log X)^B}\mu(d)^2\tau(d)^J
       |A_d(X)-g(d)\Li(X)|
 \ll_{A,J}\frac{X}{(\log X)^A}.
\end{equation}
\end{proposition}

\begin{proof}
Take the weight \(m_X(\alpha)\) to be the indicator that
\(\alpha=(\pi-1)/b\) for \(\pi\in\PP\), \(N\pi\le X\).
Its support consists of distinct nonzero Gaussian integers
with \(N\alpha\ll X\).
The divisibility condition \(\fa\mid(\alpha)\) is equivalent to
\(\pi\equiv1\pmod{b\fa}\). Equation~\eqref{eq:ray-bv} and
the preceding deletion estimate give
\eqref{eq:ideal-hypothesis} with
\[
 Y_X=\Li(X),\qquad \gamma(\fa)=\frac4{\Phi(b\fa)}
       \ll\frac{\tau(N\fa)^2}{N\fa}.
\]
The factor \(Nb=8\) is absorbed in the logarithmic cutoff.
For odd ideals, \(\gamma(\fa)=1/\Phi(\fa)\); adjoining
a single factor \((1+i)\) multiplies the density by \(1/2\).
The local identity \eqref{eq:exact-transfer} therefore yields
\eqref{eq:density}. Theorem~\ref{thm:transfer} proves the result.
\end{proof}

\section{Distribution of the switched sequence}

Put
\[
 z=X^{1/8},\qquad y=X^{1/3},
\]
and define the set of Gaussian integers
\begin{equation}\label{eq:switched-set}
 \WW_X=\left\{b\varepsilon\pi_1\pi_2\pi_3:
 \begin{array}{l}
 \varepsilon\in\{1,-1,i,-i\},\quad\pi_j\in\PP,\\
 z\le N\pi_3<y<N\pi_2<N\pi_1,\quad
 8N\pi_1N\pi_2N\pi_3\le X
 \end{array}\right\}.
\end{equation}
Unique factorization and the strict order of the three norms
give each element exactly one representation in
\eqref{eq:switched-set}. Thus \(\WW_X\) has multiplicity one,
and \(|\WW_X|\ll X\) by lattice counting.

\begin{theorem}\label{thm:switching}
For every \(A>0\), there is \(B=B(A)>0\) such that, with
\(Q=\sqrt X/(\log X)^B\),
\begin{equation}\label{eq:switched-ideal}
 \sum_{N\fa\le Q}\max_{(\alpha,\fa)=1}
 \left|
 \#\{\omega\in\WW_X:\omega\equiv\alpha\pmod\fa\}
 -\frac{\#\{\omega\in\WW_X:(\omega,\fa)=1\}}{\Phi(\fa)}
 \right|
 \ll_A\frac{X}{(\log X)^A}.
\end{equation}
\end{theorem}

The estimate is the ideal-modulus input for the switching
argument; compare \cite[Proposition~5.1]{IwaniecJimenezCM}.
We establish it before applying Theorem~\ref{thm:transfer}
to rational divisibility.

\subsection{Analytic inputs}

We shall use
\begin{equation}\label{eq:elementary-ideal-sums}
 \#\{\fa:N\fa\le U\}\ll U,\qquad
 \sum_{N\fa\le U}\frac1{\Phi(\fa)}\ll(\log(2U))^8.
\end{equation}
The first bound counts Gaussian generators up to the four units.
For the second, \(N\fa/\Phi(\fa)\le\tau(N\fa)^2\) and the number
of ideals of norm \(n\) is at most \(\tau(n)\); apply the
harmonic divisor moment \(\sum_{n\le U}\tau(n)^3/n\ll
(\log(2U))^8\).

For every fixed \(r,C>0\), equation~\eqref{eq:ray-bv} also gives
the small-conductor estimate
\begin{equation}\label{eq:small-character}
 \sum_{\substack{\pi\in\PP\\T\le N\pi<T'}}\chi(\pi)
 \ll_{r,C}\frac{T}{(\log X)^C},
 \qquad T\le T'\le2T,
\end{equation}
uniformly when
\[
 X^{1/8}\le T\le X^{1/3},\qquad
 1<N\ff\le(\log X)^r,\qquad \ff\text{ odd},
\]
and \(\chi\bmod\ff\) is a primitive nonprincipal character.
Indeed, apply \eqref{eq:ray-bv} at scale \(2T\) to the
modulus \(b\ff\), subtract the two endpoint estimates with an
\(O(1)\) adjustment for the half-open interval, and sum the
reduced residue classes with the character weights.
Their number is bounded by a fixed power of \(\log X\).
Since \(\ff\) is odd, the Chinese remainder theorem permits
every invertible class modulo \(\ff\) together with the
primary condition modulo \(b\). The main terms cancel by
\(\sum_{\xi\bmod\ff}\chi(\xi)=0\).
The deletion of degree-two prime ideals costs \(O(\sqrt T)\).
Increasing the logarithmic saving in \eqref{eq:ray-bv}
proves \eqref{eq:small-character}.

The second input is the Gaussian multiplicative large sieve:
\begin{equation}\label{eq:large-sieve}
 \sum_{N\ff\le U}\frac{N\ff}{\Phi(\ff)}
 \sum_{\chi\bmod\ff}^{*}
 \left|\sum_{0<N\alpha\le N}a_\alpha\chi(\alpha)\right|^2
 \ll (U^2+N)\sum_{0<N\alpha\le N}|a_\alpha|^2.
\end{equation}
The star denotes primitive characters of
\((\OO/\ff)^\times\), extended by zero on nonunits.
This is Hinz \cite[(3.1), p.~180]{HinzGaussian},
specialized to \(\mathbf Q(i)\) and auxiliary ideal \(\OO\).
Coprimality restrictions can be imposed by setting the
corresponding coefficients to zero.

\subsection{Proof of the ideal distribution estimate}

\begin{proof}[Proof of Theorem~\ref{thm:switching}]
For an even ideal, both counts in \eqref{eq:switched-ideal}
are zero, so we consider odd moduli.
Choose \(\delta=(\log X)^{-d_0}\), with \(d_0\) to be fixed.
Partition \([z,y)\) into the intervals
\[
 I_T=[T,\min\{(1+\delta)T,y\}),\qquad
 T=z(1+\delta)^k,\qquad k\in\mathbf Z_{\ge0},\quad T<y.
\]
Their number satisfies \(H\ll(\log X)^{d_0+1}\).
In the block \(N\pi_3\in I_T\), replace the product cutoff
by \(8N(\pi_1\pi_2)T\le X\).
Let \(\WW'_X\) be the union of the resulting sets.
Representations remain unique, and
\[
 \WW_X\subseteq\WW'_X,\qquad
 \WW'_X\setminus\WW_X
 \subseteq\{\omega:X<N\omega\le(1+\delta)X\}.
\]
For every ideal \(\fa\) and every residue class, the number
of lattice points in this annulus is
\[
 O\left(\frac{\delta X}{N\fa}
       +\frac{\sqrt X}{\sqrt{N\fa}}+1\right).
\]
This follows by counting points in the two disks using the
area and perimeter bounds for a rotated square lattice of
covolume \(N\fa\).
Summing over \(N\fa\le Q\) by
\eqref{eq:elementary-ideal-sums} gives
\[
 O(\delta X\log X+\sqrt{XQ}+Q).
\]
The change in the coprime main term is at most
\(O((\delta X+\sqrt X)(\log X)^8)\).
Thus replacing \(\WW_X\) by \(\WW'_X\) costs
\begin{equation}\label{eq:boundary-error}
 O((\delta X+X^{3/4})(\log X)^8)
\end{equation}
for \(Q\le\sqrt X\). The boundary is estimated for the union
as a whole, so this error is not multiplied by \(H\).

Fix a block \(I_T\). Let \(a(\alpha)\) and \(c_T(\beta)\)
be the indicators, respectively, of
\[
 \alpha=\pi_3\in\PP,\quad N\pi_3\in I_T,
\]
and
\[
 \beta=\pi_1\pi_2,\quad
 \pi_1,\pi_2\in\PP,\quad
 y<N\pi_2<N\pi_1,\quad N\beta\le X/(8T).
\]
The factorization defining \(\beta\) is unique. Consequently
\begin{equation}\label{eq:energies}
 \sum_\alpha|a(\alpha)|^2\ll T,\qquad
 \sum_\beta|c_T(\beta)|^2\ll X/T.
\end{equation}
Character orthogonality bounds the sum of discrepancies of
this block by
\[
 4\sum_{\substack{N\fa\le Q\\\fa\ {\rm odd}}}
       \frac1{\Phi(\fa)}
       \sum_{\substack{\chi\bmod\fa\\\chi\ne\chi_0}}
       |A_\fa(\chi)C_\fa(\chi)|,
\]
where the character sums use \(a\) and \(c_T\).
The factor four accounts for the units in
\eqref{eq:switched-set}; \(b\) is invertible modulo \(\fa\).
Induce each character from its primitive conductor \(\ff\)
and write \(\fa=\ff\fm\).
Since \(\Phi(\ff\fm)\ge\Phi(\ff)\Phi(\fm)\), enlarging
the positive summation range gives
\begin{equation}\label{eq:induced-sum}
 4\sum_{\substack{N\fm\le Q\\\fm\ {\rm odd}}}\frac1{\Phi(\fm)}
  \sum_{\substack{N\ff\le Q\\\ff\ {\rm odd}}}\frac1{\Phi(\ff)}
  \sum_{\substack{\chi\bmod\ff\ {\rm primitive}\\\chi\ne\chi_0}}
       |A_\fm(\chi)C_\fm(\chi)|.
\end{equation}
Here
\[
 A_\fm(\chi)=\sum_{(\alpha,\fm)=1}a(\alpha)\chi(\alpha),
 \qquad
 C_\fm(\chi)=\sum_{(\beta,\fm)=1}c_T(\beta)\chi(\beta).
\]
Extension of a primitive character by zero already accounts
for coprimality with \(\ff\); coprimality of \(\ff\) and
\(\fm\) is not required.

Put \(R=(\log X)^r\). For \(N\ff\le R\), at most four
primary primes occurring in \(a\) can divide \(\fm\), since
their norms are at least \(X^{1/8}\) and \(N\fm\le X^{1/2}\).
Therefore, uniformly in \(\fm\),
\[
 A_\fm(\chi)=\sum_\alpha a(\alpha)\chi(\alpha)+O(1).
\]
Use \eqref{eq:small-character} and
\(|C_\fm(\chi)|\ll X/T\).
The number of primitive characters modulo \(\ff\) is at most
\(\Phi(\ff)\), and there are \(O(R)\) ideals of norm at most
\(R\). The small-conductor contribution to
\eqref{eq:induced-sum} is bounded by
\begin{equation}\label{eq:small-conductors}
 \ll R(\log X)^8
       \left(\frac{X}{(\log X)^C}+\frac Xz\right).
\end{equation}

For \(U<N\ff\le2U\), \(U\ge R\), Cauchy's inequality,
\eqref{eq:large-sieve}, and \eqref{eq:energies} give,
uniformly in \(\fm\),
\begin{align*}
 &\sum_{U<N\ff\le2U}\frac1{\Phi(\ff)}
       \sum_{\chi\bmod\ff}^{*}
           |A_\fm(\chi)C_\fm(\chi)|\\
 &\quad\ll
 \frac1U\bigl((U^2+T)T\bigr)^{1/2}
          \bigl((U^2+X/T)X/T\bigr)^{1/2}\\
 &\quad\ll
 \sqrt X\left(U+\sqrt T+\sqrt{X/T}+\frac{\sqrt X}{U}\right).
\end{align*}
Summing dyadic intervals up to \(Q\) and then summing
\(1/\Phi(\fm)\) bounds the large-conductor part by
\begin{equation}\label{eq:large-conductors}
 \ll\sqrt X(\log X)^9
       \left(Q+\sqrt y+\sqrt{X/z}+\frac{\sqrt X}{R}\right).
\end{equation}
After multiplication by \(\sqrt X\), the middle two
terms have sizes \(X^{2/3}\) and \(X^{15/16}\).

Given \(A>0\), one possible choice is
\[
 d_0=A+20,\qquad r=A+d_0+20,\qquad
 B=A+d_0+20,\qquad C=A+d_0+r+20.
\]
With these choices, \eqref{eq:boundary-error} and the sums
over \(H\ll(\log X)^{d_0+1}\) blocks of
\eqref{eq:small-conductors} and \eqref{eq:large-conductors}
are \(O_A(X/(\log X)^A)\).
The fixed powers of \(\log X\) multiplying \(X^{7/8}\),
\(X^{15/16}\), and smaller powers are absorbed for large \(X\).
This proves \eqref{eq:switched-ideal}.
\end{proof}

\subsection{Rational divisibility in the switched sequence}

\begin{lemma}\label{lem:coprimality}
For every fixed integer \(J\ge0\) and \(Q\le\sqrt X\),
\begin{equation}\label{eq:coprimality-removal}
 \sum_{\substack{N\fa\le Q\\\fa\ {\rm odd}}}
 \frac{\tau(N\fa)^J}{\Phi(\fa)}
 \#\{\omega\in\WW_X:(\omega,\fa)\ne1\}
 \ll_J X^{7/8}(\log X)^{C_J}.
\end{equation}
Consequently the ideal-divisibility counts of \(1+\omega\)
have the form
\[
 B_X(\fa)=|\WW_X|\gamma_B(\fa)+R_B(\fa),\qquad
 \gamma_B(\fa)=
 \begin{cases}
 1/\Phi(\fa),&\fa\text{ odd},\\
 0,&\fa\text{ even},
 \end{cases}
\]
and satisfy \eqref{eq:ideal-hypothesis}.
\end{lemma}

\begin{proof}
Each \(\omega\) has three odd prime ideal factors, all of norm
at least \(z\). For each such \(\fl\),
\[
 \sum_{\substack{N\fa\le Q\\\fl\mid\fa}}
       \frac{\tau(N\fa)^J}{\Phi(\fa)}
 \ll_J\frac1{N\fl}
       \sum_{N\fm\le Q}\frac{\tau(N\fm)^J}{\Phi(\fm)}
 \ll_J\frac{(\log X)^{C_J}}{N\fl}.
\]
Indeed, writing \(\fa=\fl\fm\) changes \(\Phi\) by a factor
at least \(N\fl/2\) and the divisor weight by a bounded
factor. The remaining sum is a harmonic divisor moment.
Summing first over \(\omega\), using \(|\WW_X|\ll X\),
proves \eqref{eq:coprimality-removal}.

For an odd ideal, \(\fa\mid(1+\omega)\) is the reduced
residue condition \(\omega\equiv-1\pmod\fa\).
Theorem~\ref{thm:switching} gives the main term
\(\#\{\omega:(\omega,\fa)=1\}/\Phi(\fa)\).
Equation~\eqref{eq:coprimality-removal} with \(J=0\)
replaces it by \(|\WW_X|/\Phi(\fa)\).
For even ideals the count and the stated main term both
vanish, since \(1+\omega\) is odd.
\end{proof}

\begin{proposition}\label{prop:second-distribution}
Define a multiplicative squarefree density \(g_B\) by
\[
 g_B(2)=0,\qquad g_B(\ell)=g(\ell)\quad(\ell\text{ odd}).
\]
For every \(A>0\) and integer \(J\ge0\), some \(B=B(A,J)\)
satisfies
\begin{equation}\label{eq:second-distribution}
 \sum_{d\le\sqrt X/(\log X)^B}\mu(d)^2\tau(d)^J
 \left|\#\{\omega\in\WW_X:d\mid N(1+\omega)\}
                -g_B(d)|\WW_X|\right|
 \ll_{A,J}\frac{X}{(\log X)^A}.
\end{equation}
\end{proposition}

\begin{proof}
Apply Theorem~\ref{thm:transfer} with
\(m_X(\beta)=\one_{\beta-1\in\WW_X}\).
This is an indicator on nonzero Gaussian integers with
\(N\beta\ll X\). Lemma~\ref{lem:coprimality} supplies the
ideal distribution estimate, and
\(\gamma_B(\fa)\ll\tau(N\fa)^2/N\fa\).
The split, inert, and ramified local terms in
\eqref{eq:exact-transfer} give precisely the stated density.
\end{proof}

\section{The weighted sieve and the \texorpdfstring{\(P_2\)}{P2} theorem}

Let \(\SeqA_X\) be the multiset
\[
 \SeqA_X=\{N((\pi-1)/b):\pi\in\PP,\ N\pi\le X\}.
\]
Its members are at most \(X\) for sufficiently large \(X\).
It counts each rational prime \(p\le X\), \(p\equiv1\pmod4\), twice,
through its two primary Frobenius eigenvalues.
For a multiset \(\mathcal C\) of positive integers, define
\[
 P(t)=\prod_{\substack{\ell<t\\\ell\ {\rm prime}}}\ell,\qquad
 S(\mathcal C,t)=\#\{n\in\mathcal C:(n,P(t))=1\},
\]
and let \(\mathcal C_d\) be its submultiset of multiples of \(d\).

\subsection{Sieve densities and the Chen weight}

For the density \(g\) in \eqref{eq:density}, put
\[
 V(t)=\prod_{\ell<t}(1-g(\ell)).
\]
For split odd primes \(g(\ell)=2/\ell+O(\ell^{-2})\), and
for inert primes \(g(\ell)=O(\ell^{-2})\).
The quantitative prime number theorem in the two reduced residue classes
modulo \(4\) gives the dimension-one conditions
\begin{equation}\label{eq:sieve-conditions}
 0\le g(\ell)\le\frac12,\qquad
 \sum_{w\le\ell<t}g(\ell)\log\ell
       =\log(t/w)+O(1)\quad(2\le w<t),
\end{equation}
and
\begin{equation}\label{eq:sieve-product}
 V(t)\sim\frac{C_0}{\log t}\qquad(C_0>0).
\end{equation}
The density \(g_B\) satisfies \eqref{eq:sieve-conditions}.
Its product \(V_B(t)=\prod_{\ell<t}(1-g_B(\ell))\) equals
\(2V(t)\) for \(t>2\), and hence \(V_B(t)\sim2C_0/\log t\).

We use the dimension-one linear sieve of Jurkat and Richert
\cite{JurkatRichert}, in the form of
\cite[Theorems~11.12--11.13]{Richert1976}.
The relevant sieve functions are
\begin{equation}\label{eq:sieve-functions}
 f(s)=\frac{2e^\gamma\log(s-1)}s\quad(2\le s\le4),
 \qquad F(s)=\frac{2e^\gamma}s\quad(1\le s\le3),
\end{equation}
where \(\gamma\) is Euler's constant; see
\cite[(11.73)--(11.74)]{Richert1976}.
The upper estimate includes the remainder
\(\sum_{d<D,\ d\mid P(t)}3^{\Omega(d)}|R_d|\).
The estimates \eqref{eq:first-distribution} and
\eqref{eq:second-distribution} include all the fixed divisor
weights required below.

For a squarefree integer \(n\le X\) with \((n,P(z))=1\),
let \(\nu_y(n)\) be the number of prime factors in \([z,y]\),
and let \(\eta(n)\) indicate the pattern
\begin{equation}\label{eq:three-factor-pattern}
 n=q_1q_2q_3,\qquad z\le q_3\le y<q_2<q_1,
 \qquad q_1,q_2,q_3\text{ prime}.
\end{equation}
Chen's weight in this setting is
\[
 w(n)=1-\frac12\nu_y(n)-\frac12\eta(n).
\]
If \(\Omega(n)\ge4\), at least two factors are at most \(y\),
since three factors exceeding \(y=X^{1/3}\) already have
product greater than \(X\). Thus \(w(n)\le0\).
If \(\Omega(n)=3\), either two factors are at most \(y\),
or the pattern \eqref{eq:three-factor-pattern} holds; again
\(w(n)\le0\). For \(\Omega(n)\le2\), one has \(0\le w(n)\le1\).
Positive weights therefore detect squarefree \(P_2\)'s.

Removing nonsquarefree \(z\)-rough elements of \(\SeqA_X\)
costs
\begin{equation}\label{eq:squareful-error}
 \ll\sum_{\substack{\ell\ge z\\\ell\ {\rm prime}}}
           \sum_{m\le X/\ell^2}r_2(\ell^2m)
 \ll X\log X\sum_{\ell\ge z}\ell^{-2}
 \ll \frac{X\log X}{z}.
\end{equation}
For the middle bound, use \(r_2(n)\le4\tau(n)\) and
\(\tau(\ell^2m)\le3\tau(m)\).
A \(z\)-rough integer at most \(X\) has at most eight prime
factors, counted with multiplicity. In
\(W_1-\tfrac12W_2-\tfrac12W_3\) below, each removed occurrence
therefore has bounded coefficient: it contributes once to \(W_1\),
at most eight times to \(W_2\), and never to \(W_3\).

Set
\[
 W_1=S(\SeqA_X,z),\qquad
 W_2=\sum_{\substack{z\le q\le y\\q\ {\rm prime}}}
                      S((\SeqA_X)_q,z),
\]
and let \(W_3\) count the elements of \(\SeqA_X\) having
the pattern \eqref{eq:three-factor-pattern}.
If \(G(X)\) is the number of squarefree \(z\)-rough \(P_2\)'s
in \(\SeqA_X\), then
\begin{equation}\label{eq:chen-inequality}
 G(X)\ge W_1-\frac12W_2-\frac12W_3
                -O(X\log X/z).
\end{equation}
This is the weighted decomposition used in
\cite[\S\S2--5]{IwaniecJimenezCM}.

\subsection{The first two sieve terms}

Choose \(D=\sqrt X/(\log X)^B\), with \(B\) large enough
for the two distribution estimates and the required divisor
weights.
Proposition~\ref{prop:first-distribution} and the lower
linear sieve \cite[Theorem~11.13, (11.84)]{Richert1976}
give
\begin{equation}\label{eq:W1}
 W_1\ge
 \left(\frac{e^\gamma}{2}\log3-o(1)\right)\Li(X)V(z).
\end{equation}
To match the main-mass parameter in the cited theorem, put
\[
 Y=\Li(X),\qquad
 D_Y=\frac{\sqrt Y}{(\log Y)^{A_3}}\le D
\]
with \(A_3\) sufficiently large.
Equation~\eqref{eq:first-distribution}, with a sufficiently
strong logarithmic saving, gives the required averaged
remainder bound at this cutoff. The argument of \(f\) in the
cited theorem is
\(\tfrac12\log Y/\log z=4-o(1)\).

For \(z\le q\le y\), apply the upper sieve at level \(D/q\)
to \((\SeqA_X)_q\).
Since \(q\) does not divide \(P(z)\), the main term is
\(\Li(X)g(q)g(d)\) for \(d\mid P(z)\).
The total remainder over \(q\) and \(d\) is bounded by
\eqref{eq:first-distribution}: with \(e=qd\), its
coefficient is at most \(\tau(e)^3\).
The upper parameter is
\[
 s_q=\frac{\log(D/q)}{\log z}
       =4-\frac{8\log q}{\log X}-o(1).
\]
The error is uniform for \(z\le q\le y\), and \(s_q\in[1,3]\)
throughout this range for sufficiently large \(X\).
The upper estimate in
\cite[Theorem~11.12, (11.79), (11.81)]{Richert1976}
therefore applies throughout this range.
Partial summation over split primes, whose density is
\(1/2\), gives
\begin{align}
 W_2
 &\le\left(2e^\gamma
       \int_{1/8}^{1/3}\frac{dt}{t(4-8t)}+o(1)\right)
                   \Li(X)V(z)\notag\\
 &=\left(\frac{e^\gamma}{2}\log6+o(1)\right)\Li(X)V(z).
 \label{eq:W2}
\end{align}
Here \(g(q)=2/q+O(q^{-2})\) on split primes.
The inert contribution is negligible since
\(\sum_{q\ge z,\ q\equiv3\ (4)}g(q)\ll z^{-1}\).

\subsection{The switched sieve term}

An inert prime occurs to even valuation in a Gaussian norm.
Since the pattern \eqref{eq:three-factor-pattern} is
squarefree and \(z>2\), its three rational primes split
in \(\mathbf Z[i]\).
For each element counted by \(W_3\), factor its Gaussian
representative to write
\[
 \pi=1+\omega,\qquad
 \omega=b\varepsilon\pi_1\pi_2\pi_3.
\]
Unique factorization, the primary convention, and the strict
order of the norms determine the representation uniquely.
If \(N\pi\le X\), then
\[
 N\omega\le(\sqrt X+1)^2.
\]
Replacing this cutoff by \(N\omega\le X\) omits
\(O(\sqrt X)\) elements, since the annulus has bounded
radial thickness. The endpoint \(N\pi_3=y\) is empty
under either cutoff for large \(X\), because it would
force \(N\omega>8y^3=8X\).
It follows that
\begin{equation}\label{eq:W3-primes}
 W_3\le\#\{\omega\in\WW_X:N(1+\omega)\text{ prime}\}
                     +O(\sqrt X).
\end{equation}
For all \(\omega\in\WW_X\),
\[
 N\omega\ge8zy^2=8X^{19/24},\qquad
 N(1+\omega)\gg X^{19/24}.
\]
Thus every prime value in \eqref{eq:W3-primes} survives
sifting at \(z_0=X^{1/4}\).

Let \(\BB_X\) be the multiset of \(N(1+\omega)\) with
\(\omega\in\WW_X\). Its sifting product satisfies
\[
 V_B(z_0)=2V(z_0)=(1+o(1))V(z)
\]
by \eqref{eq:sieve-product}.
Apply Proposition~\ref{prop:second-distribution} and the
upper sieve with parameter \(\log D/\log z_0=2-o(1)\).
The same upper-sieve theorem used for \(W_2\) yields
\begin{equation}\label{eq:W3}
 W_3\le(e^\gamma+o(1))V(z)|\WW_X|
                     +O_A(X/(\log X)^A).
\end{equation}

To estimate \(|\WW_X|\), count the four units and the
ordered primary prime factors. Dropping the lower bound
\(N\pi_1>N\pi_2\) in the innermost count gives
\[
 |\WW_X|\le
 4\sum_{\substack{\pi_3\in\PP\\z\le N\pi_3<y}}
   \sum_{\substack{\pi_2\in\PP\\
              y<N\pi_2<\sqrt{X/(8N\pi_3)}}}
       \Pi\left(\frac{X}{8N\pi_2N\pi_3}\right).
\]
All nonempty inner arguments exceed \(y\), so the asymptotic
\(\Pi(T)\sim\Li(T)\) is uniform here.
Two partial summations, with
\(u=\log N\pi_2/\log X\) and \(v=\log N\pi_3/\log X\),
give
\begin{equation}\label{eq:switched-mass}
 |\WW_X|\le\left(\frac c2+o(1)\right)\Li(X),\qquad
 c=\int_{1/8}^{1/3}\int_{1/3}^{(1-v)/2}
                       \frac{du\,dv}{uv(1-u-v)}.
\end{equation}
The fixed factor \(8\) perturbs the integration boundaries
by \(O(1/\log X)\). The factor \(1/2\) in the main term
is the four units divided by \(Nb=8\); \(\Pi\) already
counts both conjugates.
Integrating in \(u\) gives
\begin{equation}\label{eq:c-integral}
 c=\int_{1/8}^{1/3}\frac{\log(2-3v)}{v(1-v)}\,dv.
\end{equation}

We use the classical bound in \cite[(10.79)]{Richert1976}.
For completeness, \(\log x\le(x-x^{-1})/2\) for \(x\ge1\).
Taking \(x=2-3v\) gives
\[
 \frac{\log(2-3v)}{v(1-v)}
 \le\frac34\left(\frac1v-\frac3{2-3v}\right).
\]
Consequently
\begin{equation}\label{eq:positive-constant}
 c\le\frac34\log\frac{64}{39}<\log\frac32.
\end{equation}
The strict inequality is equivalent to
\(16\cdot64^3<81\cdot39^3\).

\subsection{Completion of the proof}

\begin{proof}[Proof of Theorem~\ref{thm:main}]
Combining \eqref{eq:chen-inequality}, \eqref{eq:W1},
\eqref{eq:W2}, \eqref{eq:W3}, and
\eqref{eq:switched-mass} gives
\[
 G(X)\ge
 \left(\frac{e^\gamma}{4}
              \left(\log\frac32-c\right)-o(1)\right)
                   \Li(X)V(z)
 \gg\frac{X}{(\log X)^2}
\]
by \eqref{eq:positive-constant} and
\eqref{eq:sieve-product}.
The error in \eqref{eq:squareful-error} is smaller than
this lower bound. Each rational prime contributes two
primary Gaussian primes, so division by two gives the
asserted quantitative lower bound.

All counted \(n_p\) are \(z\)-rough, hence odd.
For each counted prime, choose a primary Frobenius eigenvalue
\(\pi=a+vi\).
The primary congruence implies
\[
 a\text{ odd},\qquad v\text{ even},\qquad a+v\equiv1\pmod4.
\]
By \eqref{eq:cm-norm},
\[
 n_p=\frac{(a-1)^2+v^2}{8}.
\]
If \(v\equiv0\pmod4\), both squares in the numerator
vanish modulo \(16\), so \(n_p\) is even.
If \(v\equiv2\pmod4\), the numerator is \(8\pmod{16}\),
so \(n_p\) is odd. Thus oddness is equivalent to
\(v\equiv2\pmod4\), equivalently
\(p=a^2+v^2\equiv5\pmod8\).
The roughness assertion is part of the definition of \(G(X)\).
\end{proof}

\begin{corollary}\label{cor:group}
For the primes counted in Theorem~\ref{thm:main}, every
prime divisor of \(n_p\) is congruent to \(1\pmod4\), and
\[
 E(\mathbf F_p)\simeq
       \mathbf Z/2\mathbf Z\times\mathbf Z/(4n_p)\mathbf Z.
\]
\end{corollary}

\begin{proof}
The integer \(n_p\) is an odd squarefree Gaussian norm,
so every prime factor splits in \(\mathbf Z[i]\).
For the group structure, write
\(E(\mathbf F_p)\simeq\mathbf Z/d\mathbf Z\times
\mathbf Z/e\mathbf Z\), with \(d\mid e\);
see \cite[p.~816]{IwaniecJimenezCM}.
The four points of \(E[2]\) are rational, hence \(2\mid d\).
Since \(d^2\mid8n_p\) and \(n_p\) is odd and squarefree,
one has \(d=2\) and \(e=4n_p\).
\end{proof}

\newpage

\section*{Declaration}
GPT was used in generating proofs.

\bibliographystyle{amsplain}
\bibliography{Gaussian_Norms_CM_Almost_Prime_Orders}

\end{document}